\pdfoutput=1
\documentclass[11pt]{article}

\usepackage[a4paper,margin=1in]{geometry}
\usepackage[T1]{fontenc}
\usepackage{lmodern}
\usepackage{amsmath,amssymb,amsthm,mathtools}
\usepackage{enumitem}
\usepackage{booktabs}
\usepackage{array}
\usepackage{float}
\usepackage{cite}
\usepackage{url}
\usepackage{xurl}
\usepackage[hidelinks]{hyperref}
\hypersetup{
  pdftitle={Soft Metric Spaces via Soft Elements: Topology Across Parameter-Cardinalities and Fixed-Point Theory},
  pdfauthor={Subhasis Ray},
  pdfsubject={Soft metric spaces, product uniformity, and fixed-point theory},
  pdfkeywords={soft metric space, soft element, product uniformity, fixed point theorem}
}
\usepackage{microtype}

\theoremstyle{plain}
\newtheorem{theorem}{Theorem}[section]
\newtheorem{lemma}[theorem]{Lemma}
\newtheorem{proposition}[theorem]{Proposition}
\newtheorem{corollary}[theorem]{Corollary}

\theoremstyle{definition}
\newtheorem{definition}[theorem]{Definition}
\newtheorem{example}[theorem]{Example}

\theoremstyle{remark}
\newtheorem{remark}[theorem]{Remark}

\newcommand{\Pow}{\mathcal{P}}

\newcommand{\N}{\mathbb{N}}
\newcommand{\SE}{\mathrm{SE}}
\newcommand{\dsup}{d_{\infty}}
\newcommand{\dPi}{d_{\Pi}}
\newcommand{\bared}[1]{\bar d_{#1}}

\newcommand{\tprod}{\tau_{\mathrm{prod}}}
\newcommand{\Uprod}{\mathcal{U}_{\mathrm{prod}}}

\numberwithin{equation}{section}
\numberwithin{figure}{section}
\numberwithin{table}{section}
\allowdisplaybreaks
\title{Soft Metric Spaces via Soft Elements:\\
Topology Across Parameter-Cardinalities and Fixed-Point Theory}
\author{Subhasis Ray\\
\small Department of Mathematics, Visva-Bharati\\
\small Santiniketan 731235, West Bengal, India\\
\small \texttt{subhasis.ray@visva-bharati.ac.in}\\
\small ORCID: 0000-0003-1100-4632}
\date{July 2026}

\begin{document}
\maketitle

\begin{abstract}
A soft set $(F,E)$ determines the selection space $\SE(F)=\prod_{e\in E}F(e)$. This paper studies two natural structures on that space. For at most countable $E$, the series metric $\dPi$ induces the product topology. For arbitrary $E$, the sup metric $\dsup$ induces uniform convergence. We prove that these metric spaces are complete exactly when all fibres are complete. We then compare global contractions with coordinatewise contractions and give counterexamples when coordinate separability or a uniform contractive bound is absent. Standard fixed-point theorems for complete metric spaces are recorded as direct consequences, without repeating their classical proofs. For uncountable $E$, the product topology may fail to be metrizable, so we work with its product uniformity. A parameterwise contraction theorem gives a unique fixed point and convergence in the product topology; a common bound below one gives convergence in $\dsup$.
\end{abstract}


\section{Introduction}

Soft set theory was introduced by Molodtsov as a parameterized model for uncertain information \cite{Molodtsov1999}. Maji, Biswas and Roy developed basic operations on soft sets \cite{MajiBiswasRoy2003}. Ali et al. later corrected several of those operations and clarified their algebraic properties \cite{Ali2009}. Soft topological spaces were studied by Shabir and Naz \cite{ShabirNaz2011}, while \c{C}a\u{g}man, Karata\c{s} and Engino\u{g}lu developed another form of soft topology \cite{CagmanKaratasEnginoglu2011}.

Das and Samanta introduced soft real numbers and treated soft elements as parameter-indexed selections \cite{DasSamantaSoftReal2012}. They then defined soft metrics on soft points \cite{DasSamantaSoftMetric2013}. Chen and Lin proposed a soft Meir--Keeler fixed-point theorem \cite{ChenLin2015}. Abbas, Murtaza and Romaguera later showed that this result is not valid as stated and gave a corrected finite-parameter formulation. They also proved that a soft metric on soft points induces a compatible ordinary metric when the parameter set is finite \cite{AbbasMurtazaRomaguera2016}.

Soft fixed-point theory has developed in several directions. Mohammed and Azam considered soft-set-valued mappings and an application to a delay differential equation \cite{MohammedAzam2019}. Wadkar, Bhardwaj and Sharraf proved coupled fixed-point results in soft metric spaces \cite{WadkarBhardwajSharraf2020}. Bhardwaj et al. studied compatibility, continuity and common fixed points \cite{BhardwajEtAl2021}. Demir, \"Ozbak\i r and Y\i ld\i z worked with se-uniform spaces and a soft $E$-distance \cite{DemirEtAl2016}. Soylu and \c{C}er\c{c}i gave a classical metrization of soft metric topology \cite{SoyluCerci2024}. Aras G\"und\"uz, Bayramov and Erdem Co\c{s}kun considered contractions in soft parametric metric spaces \cite{ArasGunduzEtAl2024}. Chen and Huang introduced soft sup-comparable contractions \cite{ChenHuang2026}.

Other developments include convexity in soft sets \cite{SalihSabir2019} and sequences and series of soft complex numbers \cite{RahmanQadirSabir2025}. Alcantud et al. provide a recent survey of the main branches of soft set theory \cite{AlcantudEtAl2024}.

The present paper uses a different point set. Ray and Goldar
represented a soft set by its set of soft elements and used this
representation in soft group theory \cite{RayGoldar2017}. Goldar
and Ray later studied topology, separation and compactness on
soft-element spaces \cite{GoldarRay2019}. More recently, Ray used
soft relations to define a uniformity on the soft-element space and
studied separation, continuity, completeness and compactness in that
setting \cite{RaySoftUniform2026}. In the present paper,
\[
\SE(F)=\prod_{e\in E}F(e)
\]
is treated as an ordinary function space whose elements choose one value at each parameter.

This viewpoint is useful for iteration because it separates two notions of convergence. The product topology records coordinatewise convergence, whereas the sup metric records uniform convergence over all parameters. These structures agree for finite $E$, but they differ in general for infinite $E$. For uncountable $E$, the product topology is often not metrizable, although its natural product uniformity is still available.

The results specific to the soft-element setting are the completeness equivalence, the comparison between global and coordinatewise contractivity, the non-metrizability result for uncountable products, and the parameterwise contraction theorem. Banach, Meir--Keeler, Kannan, Chatterjea, cyclic and Nadler results are included as consequences of completeness. 

Section~2 introduces soft elements, explains their advantages over soft points, and develops the product and uniform metric structures, including topology, convergence and completeness. The section ends with an example separating product from uniform convergence. Section~3 records the transfer principle and compares global and coordinatewise contractions. Section~4 treats uncountable parameter sets through the product uniformity. Section~5 combines the discussion and conclusion.

\section{Soft elements and their metric structure}

\subsection{Soft elements and their advantages}

Throughout, $X$ is a nonempty universe and $E$ is a nonempty parameter set.

\begin{definition}
A \emph{soft set} over $X$ with parameter set $E$ is a mapping $F:E\to\Pow(X)$.
The set $F(e)$ is the approximate value of $F$ at the parameter $e$.
\end{definition}

\begin{definition}
A \emph{soft element} of $(F,E)$ is a function $x:E\to X$ such that $x(e)\in F(e)$ for every $e\in E$.
The set of all soft elements is
\[
\begin{aligned}
\SE(F)=\{x\in X^E:\;&x(e)\in F(e)\\
&\text{for every }e\in E\}.
\end{aligned}
\]
\end{definition}

\begin{remark}
We assume $\SE(F)\neq\varnothing$ whenever a metric or fixed-point theorem is stated.
If all fibres are nonempty, this assumption follows from the Axiom of Choice.
\end{remark}

By definition,
\[
\SE(F)=\prod_{e\in E}F(e)\subseteq X^E.
\]
Thus the soft elements form an ordinary selection space.
A soft element records one value at every parameter, whereas a soft point is active at only one parameter.
This full profile makes it possible to compare coordinatewise and uniform behaviour on the same space.
The difference is illustrated in Example~\ref{ex:nonuniform}. 

\subsection{Product and uniform metrics}

Assume that each fibre $F(e)$ carries a metric $d_e$.

\begin{definition}
For $e\in E$, define
\[
\bared{e}(u,v)=\min\{1,d_e(u,v)\},\qquad u,v\in F(e).
\]
\end{definition}

\begin{lemma}\label{lem:bounded-complete}
The metrics $d_e$ and $\bared{e}$ have the same convergent and Cauchy sequences.
In particular, $(F(e),d_e)$ is complete if and only if $(F(e),\bared{e})$ is complete.
\end{lemma}

\begin{proof}
For distances below $1$, the two metrics agree. Therefore they have the same Cauchy and convergent sequences.
\end{proof}

\begin{definition}
For arbitrary $E$ and $x,y\in\SE(F)$, define
\[
\dsup(x,y)=\sup_{e\in E}\bared{e}(x(e),y(e)).
\]
\end{definition}

\begin{definition}
Assume that $E$ is at most countable. If $E=\{e_1,\ldots,e_N\}$ is finite, define
\[
\dPi(x,y)=\sum_{n=1}^{N}2^{-n}\bared{e_n}(x(e_n),y(e_n)).
\]
If $E=\{e_1,e_2,\ldots\}$ is countably infinite, use the same formula with the sum taken over $n\ge1$.
\end{definition}

\begin{proposition}\label{prop:metric-properties}
The function $\dsup$ is a metric on $\SE(F)$ for arbitrary $E$.
If $E$ is at most countable, then $\dPi$ is also a metric on $\SE(F)$.
\end{proposition}

\begin{proof}
Separation and symmetry follow from the corresponding properties of the fibre metrics.
The triangle inequality for $\dsup$ follows by taking the supremum of the coordinate triangle inequalities.
For $\dPi$, multiply each coordinate triangle inequality by $2^{-n}$ and sum.
\end{proof}

\subsection{Topology, convergence and completeness}

\begin{theorem}\label{thm:topology-comparison}
Let $\tprod$ be the subspace product topology on $\SE(F)\subseteq\prod_{e\in E}F(e)$.
Then:
\begin{enumerate}[label=\textup{(\alph*)},leftmargin=2em]
\item if $E$ is at most countable, $\dPi$ induces $\tprod$;
\item for arbitrary $E$, the topology induced by $\dsup$ is finer than $\tprod$;
\item if $E$ is finite, $\dPi$ and $\dsup$ are Lipschitz equivalent.
\end{enumerate}
\end{theorem}

\begin{proof}
For (a), a basic product neighborhood controls finitely many coordinates.
A sufficiently small $\dPi$-ball controls those coordinates because every corresponding term of the series is bounded by the whole sum.
Conversely, given a $\dPi$-ball, choose a finite initial part of the series so that its remaining tail is small, and control the finitely many initial coordinates.
This gives a product neighborhood inside the ball.

For (b), a $\dsup$-ball of radius smaller than the finitely many radii in a basic product neighborhood is contained in that neighborhood.

For (c), if $E=\{e_1,\ldots,e_N\}$, then
\[
\dPi(x,y)\le\dsup(x,y)\le 2^N\dPi(x,y).
\]
\end{proof}

\begin{proposition}\label{prop:convergence}
If $E$ is at most countable, convergence in $\dPi$ is exactly coordinatewise convergence.
For arbitrary $E$, convergence in $\dsup$ is uniform convergence with respect to the bounded fibre metrics.
\end{proposition}

\begin{proof}
The first statement follows from Theorem~\ref{thm:topology-comparison}(a). The second follows from the definition of $\dsup$.
\end{proof}

\begin{theorem}\label{thm:completeness-equivalence}
Assume $\SE(F)\neq\varnothing$.
\begin{enumerate}[label=\textup{(\alph*)},leftmargin=2em]
\item If $E$ is at most countable, $(\SE(F),\dPi)$ is complete if and only if every $(F(e),d_e)$ is complete.
\item For arbitrary $E$, $(\SE(F),\dsup)$ is complete if and only if every $(F(e),d_e)$ is complete.
\end{enumerate}
\end{theorem}

\begin{proof}
Assume first that every fibre is complete.
Let $(x_k)$ be Cauchy in $\dPi$.
For each coordinate $e_n$,
\[
2^{-n}\bared{e_n}(x_k(e_n),x_m(e_n))\le\dPi(x_k,x_m),
\]
so $(x_k(e_n))$ is Cauchy in the fibre and has a limit $x(e_n)\in F(e_n)$.
The resulting selection $x$ is the $\dPi$-limit by Proposition~\ref{prop:convergence}.

For $\dsup$, a $\dsup$-Cauchy sequence is Cauchy in every coordinate.
Let $x(e)$ be the coordinate limit.
Given $\varepsilon>0$, choose $N$ such that $\dsup(x_n,x_m)<\varepsilon$ for $m,n\ge N$.
Letting $m\to\infty$ coordinatewise gives $\bared{e}(x_n(e),x(e))\le\varepsilon$ for every $e$, hence $\dsup(x_n,x)\le\varepsilon$.

Conversely, suppose one fibre $F(e_0)$ is not complete.
Choose a Cauchy sequence $(u_k)$ in $F(e_0)$ with no limit in that fibre.
Fix a base selection $a\in\SE(F)$ and define $x_k(e_0)=u_k$ and $x_k(e)=a(e)$ for $e\ne e_0$.
Then $(x_k)$ is Cauchy in $\dsup$ and, when $E$ is at most countable, also in $\dPi$.
Any limit in $\SE(F)$ would give a limit of $(u_k)$ in $F(e_0)$, a contradiction.
\end{proof}

\begin{remark}\label{rem:incomplete-fibre}
One incomplete fibre is enough to make the selection space incomplete: all other coordinates can remain fixed while the incomplete coordinate carries a Cauchy sequence with no limit.
\end{remark}

\begin{example}\label{ex:nonuniform}
Let $E=\N$, let $F(n)=[0,1]$ with the usual metric, and define
\[
(Tx)(n)=q_nx(n),\qquad q_n=1-\frac{1}{n+1}.
\]
Each coordinate map is a contraction and has fixed point $0$.
Starting from $x_0(n)=1$, the iterates satisfy $x_k(n)=q_n^k\to0$ for every fixed $n$.
Hence $x_k$ converges to $0$ in the product topology.
However,
\[
\dsup(x_k,0)=\sup_n q_n^k=1
\]
for every $k$, so the convergence is not uniform.
The full selection records both behaviours on the same space.
\end{example}

\section{Fixed-point transfer and contraction structure}

\subsection{Classical fixed-point transfer and consequences}

\begin{proposition}\label{prop:transfer}
Let $d=\dPi$ when $E$ is at most countable, and let $d=\dsup$ for arbitrary $E$. If every fibre is complete, then $(\SE(F),d)$ is complete. Consequently, any classical fixed-point result stated for complete metric spaces may be applied to $(\SE(F),d)$ after its remaining assumptions have been verified.
\end{proposition}

\begin{proof}
The completeness statement is Theorem~\ref{thm:completeness-equivalence}. The rest follows by applying the relevant metric-space theorem.
\end{proof}

We recall the contractive conditions used below.

\begin{definition}\label{def:contractive-types}
Let $(Y,d)$ be a metric space and $T:Y\to Y$.
The map $T$ is:
\begin{enumerate}[label=\textup{(\alph*)},leftmargin=2em]
\item a Banach contraction if $d(Tx,Ty)\le qd(x,y)$ for some $q<1$;
\item a Kannan contraction if
\[
d(Tx,Ty)\le\alpha\bigl(d(x,Tx)+d(y,Ty)\bigr),\qquad \alpha<\tfrac12;
\]
\item a Chatterjea contraction if
\[
d(Tx,Ty)\le\beta\bigl(d(x,Ty)+d(y,Tx)\bigr),\qquad \beta<\tfrac12;
\]
\item a Meir--Keeler contraction if, for every $\varepsilon>0$, some $\delta>0$ satisfies
\[
\varepsilon\le d(x,y)<\varepsilon+\delta\Longrightarrow d(Tx,Ty)<\varepsilon.
\]
\end{enumerate}
\end{definition}

\begin{corollary}\label{cor:standard-fixed-points}
Under the assumptions of Proposition~\ref{prop:transfer}, the following statements hold on $\SE(F)$ in the metric $d$.
\begin{enumerate}[label=\textup{(\alph*)},leftmargin=2em]
\item A Banach contraction has a unique fixed point and every Picard iteration converges to it \cite{Banach1922}.
\item A Meir--Keeler contraction has a unique fixed point \cite{MeirKeeler1969}.
\item A Kannan contraction has a unique fixed point \cite{Kannan1968}.
\item A Chatterjea contraction has a unique fixed point \cite{Chatterjea1972}.
\item Let $A,B$ be nonempty closed subsets of $\SE(F)$ and let $T:A\cup B\to A\cup B$. If $T(A)\subseteq B$, $T(B)\subseteq A$, and
\[
d(Tx,Ty)\le qd(x,y),\qquad x\in A,\ y\in B,
\]
for some $q<1$, then $T$ has a unique fixed point in $A\cap B$ \cite{KirkSrinivasanVeeramani2003}.
\item Let $\mathcal{CB}(\SE(F))$ be the family of nonempty closed bounded subsets. If $T:\SE(F)\to\mathcal{CB}(\SE(F))$ satisfies
\[
H_d(Tx,Ty)\le qd(x,y),\qquad q<1,
\]
where $H_d$ is the Hausdorff metric, then some $x^*$ satisfies $x^*\in T(x^*)$ \cite{Nadler1969}.
\end{enumerate}
\end{corollary}

\begin{proof}
Apply Proposition~\ref{prop:transfer} to the cited metric-space results.
\end{proof}

A few further consequences of the Banach contraction principle are useful here.

\begin{proposition}\label{prop:commuting}
Assume that $(\SE(F),d)$ is complete. Let $T:\SE(F)\to\SE(F)$ be a contraction and let $S:\SE(F)\to\SE(F)$ satisfy $ST=TS$.
Then $T$ and $S$ have a common fixed point. This point is the unique fixed point of $T$.
\end{proposition}

\begin{proof}
If $Tx^*=x^*$, then $T(Sx^*)=S(Tx^*)=Sx^*$.
Thus $Sx^*$ is another fixed point of $T$ and must equal $x^*$.
\end{proof}

\begin{proposition}\label{prop:coupled}
Assume that $(\SE(F),d)$ is complete. Let $\mathcal{C}:\SE(F)\times\SE(F)\to\SE(F)$ satisfy
\[
d(\mathcal{C}(x,y),\mathcal{C}(u,v))
\le \frac{q}{2}\bigl(d(x,u)+d(y,v)\bigr)
\]
for all $x,y,u,v\in\SE(F)$, where $q<1$.
Then $\mathcal{C}$ has a unique coupled fixed point.
\end{proposition}

\begin{proof}
On $\SE(F)\times\SE(F)$ use
\[
D((x,y),(u,v))=\max\{d(x,u),d(y,v)\}.
\]
The map $G(x,y)=(\mathcal{C}(x,y),\mathcal{C}(y,x))$ is a $q$-contraction.
Its unique fixed point $(x^*,y^*)$ satisfies
\[
x^*=\mathcal{C}(x^*,y^*),\qquad y^*=\mathcal{C}(y^*,x^*).
\]
\end{proof}

\begin{proposition}\label{prop:stability}
Assume that $(\SE(F),d)$ is complete. Let $T:\SE(F)\to\SE(F)$ be a contraction with constant $q<1$ and fixed point $x^*$.
Then
\[
d(x,x^*)\le\frac{d(x,Tx)}{1-q},\qquad x\in\SE(F).
\]
If $S:\SE(F)\to\SE(F)$ is another contraction with the same constant $q$ and
\[
\sup_x d(Tx,Sx)\le\varepsilon,
\]
then their fixed points satisfy
\[
d(x_T,x_S)\le\frac{\varepsilon}{1-q}.
\]
\end{proposition}

\begin{proof}
The first estimate follows from
\[
d(x,x^*)\le d(x,Tx)+qd(x,x^*).
\]
For the second, use
\[
\begin{aligned}
d(x_T,x_S)
&\le d(Tx_T,Sx_T)+d(Sx_T,Sx_S)\\
&\le\varepsilon+q\,d(x_T,x_S).
\end{aligned}
\]
\end{proof}

\subsection{Global and coordinatewise contractions}

\begin{definition}
A map $T:\SE(F)\to\SE(F)$ is \emph{coordinate-separable} if there are maps $T_e:F(e)\to F(e)$ such that
\[
(Tx)(e)=T_e(x(e))
\]
for every $e$ and every $x\in\SE(F)$.
\end{definition}

\begin{proposition}\label{prop:global-to-coordinate}
Suppose $T$ is coordinate-separable.
If $T$ is a $q$-contraction in $\dsup$, then every $T_e$ is a $q$-contraction in $\bared{e}$.
If $E$ is at most countable and $T$ is a $q$-contraction in $\dPi$, the same conclusion holds.
Without coordinate separability, a global contraction need not give a coordinatewise contraction.
\end{proposition}

\begin{proof}
Fix $e\in E$ and $u,v\in F(e)$.
Choose a base selection $a\in\SE(F)$ and form selections $x,y$ that agree with $a$ outside $e$, with $x(e)=u$ and $y(e)=v$.
For $\dsup$,
\[
\begin{aligned}
\bared{e}(T_e(u),T_e(v))
&=\dsup(Tx,Ty)\\
&\le q\dsup(x,y)
 =q\bared{e}(u,v).
\end{aligned}
\]
For $\dPi$, if $e=e_n$, the common factor $2^{-n}$ cancels in the same argument.

For a counterexample, take $E=\{1,2\}$, $F(1)=F(2)=[0,1]$, and use the max metric.
For a fixed $q\in(0,1)$ define
\[
T(x_1,x_2)=(qx_2,0).
\]
Then $T$ is a global $q$-contraction.
However, the first output coordinate cannot satisfy an estimate in terms of $|x_1-y_1|$ alone, because it changes when $x_1=y_1$ but $x_2\ne y_2$.
\end{proof}

\begin{proposition}\label{prop:coordinate-to-global}
Suppose $T$ is coordinate-separable and
\[
\bared{e}(T_e(u),T_e(v))\le q_e\bared{e}(u,v),\qquad q_e<1.
\]
If $q=\sup_{e\in E}q_e<1$, then $T$ is a $q$-contraction in $\dsup$ and, when $E$ is at most countable, also in $\dPi$.
Without the uniform bound $\sup_{e\in E}q_e<1$, the map need not be a global contraction.
\end{proposition}

\begin{proof}
Taking the supremum gives
\[
\dsup(Tx,Ty)\le q\dsup(x,y).
\]
For countable $E$, multiply each coordinate estimate by $2^{-n}$ and sum to obtain
\[
\dPi(Tx,Ty)\le q\dPi(x,y).
\]
Example~\ref{ex:nonuniform} shows that the conclusion may fail without this bound.
\end{proof}

\section{Uncountable parameter sets and the product uniformity}\label{sec:uncountable}

\subsection{Product topology and its uniformity}

\begin{lemma}\label{lem:not-first-countable}
Assume $E$ is uncountable and every fibre $F(e)$ contains at least two points.
Then the product topology on $\SE(F)$ is not first countable and hence is not metrizable.
\end{lemma}

\begin{proof}
Fix $x\in\SE(F)$ and suppose $\{U_n:n\in\N\}$ is a local base at $x$.
For each $n$, choose a basic product neighborhood $B_n$ with $x\in B_n\subseteq U_n$.
Let $S_n$ be the finite set of coordinates restricted by $B_n$.
The union $S=\bigcup_n S_n$ is countable, so choose $e_0\in E\setminus S$.
Choose $a\in F(e_0)$ different from $x(e_0)$ and a fibre neighborhood $V$ of $x(e_0)$ that does not contain $a$.
The product neighborhood
\[
W=\{y\in\SE(F):y(e_0)\in V\}
\]
contains $x$.
Some $U_n$ must be contained in $W$.
But $B_n$ does not restrict $e_0$, so it contains a selection whose $e_0$-coordinate is $a$; this point lies in $U_n$ but not in $W$, a contradiction.
\end{proof}

\begin{remark}
If only countably many fibres are non-singleton, this obstruction disappears and the product topology may be metrizable.
\end{remark}

Although the product topology need not be metrizable, it is induced by a natural uniformity.

\begin{definition}\label{def:product-uniformity}
For $e\in E$, put
\[
\rho_e(x,y)=\bared{e}(x(e),y(e)).
\]
The product uniformity $\Uprod$ is generated by the entourages
\[
\begin{aligned}
U(S,\varepsilon)=\{(x,y):\ &\rho_e(x,y)<\varepsilon\\
&\text{for every }e\in S\},
\end{aligned}
\]
where $S\subseteq E$ is finite and $\varepsilon>0$.
\end{definition}

\begin{proposition}\label{prop:uniformity-topology}
The uniformity $\Uprod$ induces the product topology $\tprod$.
If $E$ is at most countable, $\dPi$ metrizes this uniformity.
\end{proposition}

\begin{proof}
The set $U(S,\varepsilon)[x]$ restricts the coordinates in $S$ by fibre balls.
These sets form the usual product-neighborhood base.
When $E$ is at most countable, the entourages of $\dPi$ and $\Uprod$ refine one another.
\end{proof}

\subsection{The parameterwise contraction theorem}

\begin{remark}\label{rem:star-separable}
Condition $(\star_e)$ below already forces coordinate separability. If $x(e)=y(e)$, then its right-hand side is zero, so $(Tx)(e)=(Ty)(e)$. Thus the $e$-th output depends only on the $e$-th input.
\end{remark}

\begin{theorem}\label{thm:parameterwise}
Assume $\SE(F)\neq\varnothing$ and every fibre is complete.
Let $T:\SE(F)\to\SE(F)$ satisfy, for each $e\in E$, an estimate
\[
\bared{e}\bigl((Tx)(e),(Ty)(e)\bigr)
\le q_e\bared{e}\bigl(x(e),y(e)\bigr)
\tag{$\star_e$}
\]
for all $x,y\in\SE(F)$,
where $0\le q_e<1$.
Then:
\begin{enumerate}[label=\textup{(\alph*)},leftmargin=2em]
\item $T$ has a unique fixed point $x^*\in\SE(F)$;
\item for every $x_0\in\SE(F)$, the Picard sequence $x_{n+1}=Tx_n$ converges to $x^*$ in the product topology;
\item if $\sup_{e\in E}q_e<1$, the convergence is also in $\dsup$.
\end{enumerate}
\end{theorem}

\begin{proof}
Fix $x_0\in\SE(F)$ and put $x_{n+1}=Tx_n$.
For a fixed $e$, condition $(\star_e)$ gives
\[
\bared{e}(x_{n+1}(e),x_n(e))
\le q_e^n\bared{e}(x_1(e),x_0(e)).
\]
If $m>n$, the triangle inequality yields
\[
\bared{e}(x_m(e),x_n(e))
\le \bared{e}(x_1(e),x_0(e))
\sum_{k=n}^{m-1}q_e^k.
\]
The right-hand side tends to zero, so $(x_n(e))$ is Cauchy in the complete fibre $F(e)$.
Let its limit be $x^*(e)$.
The function $x^*:e\mapsto x^*(e)$ belongs to $\SE(F)$, and $x_n\to x^*$ coordinatewise, hence in $\tprod$.

To prove fixedness, use $(\star_e)$ with $y=x^*$:
\[
\begin{aligned}
\bared{e}\bigl((Tx_n)(e),(Tx^*)(e)\bigr)
&\le q_e\bared{e}\bigl(x_n(e),x^*(e)\bigr)\\
&\longrightarrow 0.
\end{aligned}
\]
Since $Tx_n=x_{n+1}$ and $x_{n+1}(e)\to x^*(e)$, uniqueness of the fibre limit gives $(Tx^*)(e)=x^*(e)$.
This holds for every $e$, so $Tx^*=x^*$.

If $u$ and $v$ are fixed points, then
\[
\bared{e}(u(e),v(e))
\le q_e\bared{e}(u(e),v(e))
\]
for every $e$.
Because $q_e<1$, all coordinate distances vanish and $u=v$.

Finally, if $q=\sup_{e\in E}q_e<1$, taking suprema in $(\star_e)$ gives
\[
\dsup(Tx,Ty)\le q\dsup(x,y).
\]
Corollary~\ref{cor:standard-fixed-points}(a) then gives convergence in $\dsup$.
\end{proof}

\begin{corollary}\label{cor:uncountable-uniform}
For arbitrary, possibly uncountable $E$, a uniform coordinate contraction with a common constant $q<1$ has a unique fixed point in $(\SE(F),\dsup)$ and its Picard iterates converge uniformly.
\end{corollary}

\begin{proof}
Apply Theorem~\ref{thm:parameterwise}(c) and Theorem~\ref{thm:completeness-equivalence}(b).
\end{proof}

\begin{table}[H]
\centering
\caption{Topology and fixed-point behaviour across parameter cardinalities.}
\label{tab:regimes}
\scriptsize
\begin{tabular}{@{}p{1.8cm}p{3.7cm}p{2.5cm}p{5.3cm}@{}}
\toprule
$E$ & Natural regimes & Metrizability & Fixed-point interpretation\\
\midrule
Finite & product and uniform coincide & yes & ordinary metric results apply without a topological distinction\\
Countable & product via $\dPi$; uniform via $\dsup$ & both yes & coordinatewise and uniform convergence may differ\\
Uncountable & product uniformity $\Uprod$; uniform metric $\dsup$ & product usually no; uniform yes & parameterwise contractions give product convergence; a uniform bound gives $\dsup$ convergence\\
\bottomrule
\end{tabular}
\normalsize
\end{table}

\section{Discussion and conclusion}

The selection space $\SE(F)$ carries two natural structures. The product topology describes coordinatewise convergence, while $\dsup$ describes uniform convergence over the parameter set. They agree when $E$ is finite, but Example~\ref{ex:nonuniform} shows that they can differ even when $E$ is countable.

Completeness of the selection space is equivalent to completeness of all fibres. The comparison results also show that global and coordinatewise contractivity are not interchangeable: coordinate separability is needed to pass from a global contraction to fibre maps, and a uniform bound below one is needed to obtain a global contraction from coordinate estimates.

For uncountable $E$, the product topology may not be metrizable, but its product uniformity still gives the correct finite-coordinate structure. Theorem~\ref{thm:parameterwise} gives a unique fixed point and product convergence under parameterwise contractions. When the contraction constants have a common bound below one, the convergence is uniform in $\dsup$. Thus fixed-point results on soft elements should state clearly whether the product or uniform regime is intended.

\end{document}